\documentclass[12pt]{amsart}
\newcommand{\bb}{\mathbf{b}}
\newcommand{\p}{\mathbf{P}}
\newcommand{\y}{\mathbf{y}}
\newcommand{\nn}{\mathbb{N}}
\newcommand{\rr}{\mathbb{R}}
\newcommand{\zz}{\mathbb{Z}}
\newcommand{\e}{\hbox{\rm e}}

\newtheorem{theorem}{Theorem}
\newtheorem{example}[theorem]{Example}
\newtheorem{corollary}[theorem]{Corollary}
\newtheorem{definition}[theorem]{Definition}
\newtheorem{lemma}[theorem]{Lemma}
\newtheorem{rem}[theorem]{Remark}
\numberwithin{equation}{section}

\title[Certificates for integer programming]
{Certificates and relaxations for integer programming 
and the semi-group membership problem}
\author{J. B. Lasserre}
\address{LAAS-CNRS and Institute of Mathematics,
LAAS 7 Avenue du Colonel Roche, 31077 Toulouse cedex 4, France}
\email{lasserre@laas.fr}
\author{E. S. Zeron}
\address{Depto. Matem\'aticas, CIVESTAV-IPN,
Apdo.~Postal 14740, Mexico D.F. 07000, M\'exico.}
\email{eszeron@math.cinvestav.mx}

\begin{document}
\begin{abstract}
We consider integer programming and the semi-group membership 
problem. We develop and extend the approach started in 
\cite{lasserrefarkas,lasserreduality} so as to provide the following 
\textit{theorem of the alternative}: the system $b=Ax$ has no nonnegative 
integral solution $x\in\nn^n$ if and only if $p(b)<0$ for some given 
polynomial $p$. The coefficients of $p$ form a vector which lies in some 
convex cone $\Omega$, and so we characterize $\Omega$. We also provide a 
hierarchy of linear programming relaxations, where the continuous case 
$Ax=b$ describes the first relaxation in the hierarchy for $x\in\rr^n$ 
and $x\geq0$.
\end{abstract}
\maketitle

\section{Introduction}
This paper is concerned with certificates for integer programming (IP) as well as with
the semi-group membership problem. That is, given a finitely generated 
abelian group $G$ (e.g. $\zz^m$), a semi-group $G_a\subset G$ generated 
by a finite family $(a_k)_{k=1}^n\subset G$, and an element $b\in G$, we 
provide a certificate of $b\in G_a$ or $b\not\in G_a$. We build upon and 
extend previous work of \cite{lasserrefarkas,lasserreduality,book}, notably on a discrete 
Farkas lemma for IP. Among other things, we provide a 
hierarchy of linear programming relaxations (LP-relaxations) for integer programming.
The first relaxation in the hierarchy is just the usual 
LP relaxation, which then appears as a {\it first-order} (or linear) 
approximation to the discrete case, whereas usually IP is viewed as an arithmetic refinement of LP. We also provide a theorem of the 
alternative (or duality theorem) in the form of a polynomial certificate 
associated with the IP problem, and we compare with the certificate for LP obtained
by the standard Farkas lemma.

A central idea in nonconvex optimization is to replace a non convex 
(hence hard) problem with a suitable easier convex problem in some lifted
space, but at the price of increasing the dimension. For instance in the 
lift-and-project approach for polynomial optimization (e.g. 0-1 problems) 
one replaces $x\in\rr^n$ with the vector $\y=(x^\alpha)$ of all moments 
and solves some hierarchy of appropriate linear or semidefinite 
relaxations. The interested reader is referred for more details to e.g.  
Sherali and Adams \cite{sherali1,sherali2}, Lov\'asz and Schrijver 
\cite{lovasz}, and Lasserre \cite{lasserresiopt1,lasserresiopt2}; see also 
Laurent \cite{laurent} for a comparison. Of course, any IP problem can
also be modeled via polynomial equations. For example, if the entry 
$x_i$ is bounded by some integer $M$, then one may include the polynomial constraint 
$\prod_{k=0}^M(x_i-k)$ which forces $x_j$ to be an integer, and so the 
above methodology applies. However, since the degree in the constraint is 
$M$ (as opposed to $2$ in the Boolean case), the size of the first 
linear or semidefinite relaxation in the hierarchy (in e.g. 
\cite{lasserresiopt2}) is already very large because it 
includes moments up to order $M$.

Let $A\in\nn^{m\times{n}}$ be a fixed matrix. In the approach 
developed in \cite{lasserrefarkas,lasserreduality} the integral 
solution $x\in\nn^n$ to $Ax=b$ is also lifted to some $\y\in\rr^p$ 
with $p\leq{n}\prod_j(1{+}b_j)$. But this time the lifting process has a different 
meaning. Indeed there is some  very simple matrix $E\in\rr^{n\times p}$
such that $x:=E\y\in\rr^n$ is now a point in the integer 
hull of feasible solutions.
Furthermore, the vector $\y$ is a point of some polytope and 
several interpretations can be deduced from the lifting process. For 
example it was already shown in \cite{lasserrefarkas} that the lifting 
process can be used to prove that $b=Ax$ for some nonnegative integral 
vector $x\in\nn^n$ if and only if the polynomial $z^b{-}1$ has a 
nonnegative representation in the binomial ideal generated by the 
binomials $(z^{A_k}{-}1)$. But this interpretation is only one 
among many other interpretations as shown in the present paper.
 
Interestingly, one can also use the lifting process to provide a 
hierarchy of LP relaxations for the IP problem, so that 
the continuous case $\{Ax=b:x\in\rr^n,x\geq0\}$ appears as a first-order 
(or linear) approximation of the discrete case with $x\in\nn^n$. This 
hierarchy also provides us with a {\it theorem of the alternative} (or 
duality theorem) which uses a non linear polynomial, a discrete 
analogue of the celebrated Farkas Lemma in linear algebra and convex 
optimization. Recall that the Farkas Lemma provides a membership 
certificate for the convex cone $\Theta:=\{Ax:x\in\rr^n,x\geq0\}$, in the form
$b\not\in\Theta$ if and only if $\omega'b<0$ for some $\omega$ with 
$\omega'A\geq0$. In other words $b\not\in\Theta$ if and only if the {\it 
linear} polynomial $z\mapsto{p_\omega}(z):=\omega'z$ is negative when 
evaluated at $z=b$.

\vspace{1em}\subsection*{Contribution}
Every finitely generated abelian group $G$ can 
be identified with a subset of $\zz^m$ and so:

\vspace{1em}$\bullet$ 
We firstly show that the semi-group membership problem reduces to the 
existence of a nonnegative integral vector $x\in\nn^n$, solution of 
some related linear system $Ax=b$ for some nonnegative integral matrix 
$A\in\nn^{m\times n}$. Shevshenko \cite[p.~11]{Shevchenko} developed 
a similar result in the framework of additive semigroups with unity, 
contained in a finite generated abelian group.

\vspace{1em}$\bullet$ 
Set $p:=\sum_{k=1}^n\prod_{j=1}^m(1{+}b_j{-}A_{j;k})$. We next show that 
existence of a nonnegative integral solution $x\in\nn^n$ to the linear 
system $Ax=b$ reduces to the existence of a nonnegative real solution 
$\y\in\rr^p$ for a system of linear equations of the form~:
\begin{eqnarray}
\label{a1}
b_j&=&\y'A^{(e_j)},\quad\forall\;1\leq{j}\leq{m};\\
\label{a2}
b_ib_j&=&\y'A^{(e_i+e_j)},\quad\forall\;1\leq{i}\leq{j}\leq{m};\\
\nonumber\dots&=&\cdots\\
\label{am}
b_1^{z_1}\cdots{}b_m^{z_m}&=&\y'A^{(z)},\quad
\begin{pmatrix}\forall\;0\leq{z_j}\leq{b_j},\\
z_1{+}\cdots{+}z_m=\delta;\end{pmatrix}\\
\nonumber\dots&=&\cdots\\
\label{ab}
b_1^{b_1}\cdots{}b_m^{b_m}&=&\y'A^{(b)};
\end{eqnarray}
for some appropriate nonnegative integer vectors 
$A^{(e_j)},A^{(e_i+e_j)},A^{(z)}\in\nn^p$ with $z\in\nn^m$ and $z\leq{b}$. 
The parameter $\delta\geq1$ in (\ref{am}) is the degree of the monomial 
$b\mapsto{}b_1^{z_1}\cdots{}b_m^{z_m}$ in (\ref{am}). Therefore a 
certificate of $b\neq{Ax}$ for every $x\in\nn^n$ is obtained as soon as 
any subsystem of (\ref{a1})-(\ref{ab}) has no solution $\y\in\rr^p$; that 
is, one does not need to consider the entire system (\ref{a1})-(\ref{ab}).

We can index the entries of $A^{(e_j)}=(A^{(e_j)}[k,u])\in\nn^p$ 
in (\ref{a1}) in such a way that $A^{(e_j)}[k,u]=A_{j;k}$ for all $u$, 
$1\leq{j}\leq{m}$, and $1\leq{k}\leq{n}$. If we also index the entries of 
$\y=(y[k,u])\in\rr^p$ in the same way, the new vector 
$\widehat{x}=(\widehat{x}_k)\in\rr^n$ with $\widehat{x}_k:=\sum_uy[k,u]$ 
satisfies $A\widehat{x}=b$ and belongs to the integer hull of 
$\{x\in\rr^n:Ax=b,x\geq0\}$, whenever $\y$ is a solution to 
(\ref{a1})-(\ref{ab}). In this approach, LP (or the continuous case) appears as a particular "first 
order" (or "linear") approximation of IP (the discrete case). Indeed,
if one considers (\ref{a1}) alone (i.e. ignoring
(\ref{a2})-(\ref{ab}) which have nonlinear 
right-hand-sides terms $b^z$ for $\|z\|>1$) then from any nonnegative solution 
$\y$ of (\ref{a1}) one obtains a real nonnegative
solution $\widehat{x}\in\rr^n$ of $A\widehat{x}=b$, and conversely.

To construct a natural hierarchy of LP-relaxations 
for the IP feasibility problem $Ax=b$, $x\in\nn^n$, just consider 
an increasing number of equations among the system (\ref{a1})-(\ref{ab}), 
so that the last (and largest size) LP-relaxation is the whole system 
(\ref{a1})-(\ref{ab}) that describes the integer hull of the set 
$\{x\in\nn^n:Ax=b\}$. Thus, if on the one hand the discrete case is an 
arithmetic refinement of the continuous one, on the other hand the discrete case can be 
approximated via LP-relaxations of increasing sizes, and these 
relaxations are different from the lift-and-project ones described in e.g. 
(\cite{laurent}). To the best of our knowledge such a hierarchy has not 
been investigated before. Even if it is not clear at the moment whether 
this hierarchy of linear relaxations is useful from a computational 
viewpoint, it provides new insights for integer programming.

On the other hand it was already proved in 
\cite{lasserrefarkas,lasserreduality} that existence of a nonnegative 
integral solution $x\in\nn^n$ to the linear system $Ax=b$, reduces to 
the existence of a nonnegative real solution $\y\in\rr^p$ for a system 
of the form
\begin{equation}\label{laas}
(-1,0,...,0,1)'\,=\,\Theta\,\y,
\end{equation}
where $\Theta$ is some appropriated {\em network} matrix (hence totally unimodular). We show 
that the system (\ref{a1})-(\ref{ab}) can be deduced from 
(\ref{laas}) by multiplying it from the left times a square invertible 
matrix $\Delta\in\rr^{s\times{s}}$. In particular $\Delta$ is a Kronecker 
product of Vandermonde matrices. Hence any real vector $\y\geq0$ is 
solution of (\ref{laas}) if and only the same $\y$ is solution of 
(\ref{a1})-(\ref{ab}).

\vspace{1em} $\bullet$ 
We provide a polyhedral convex cone $\Omega\subset\rr^s$ associated with
(\ref{a1})-(\ref{ab}) for some $s\in\nn$, such that a direct application 
of Farkas lemma to the continuous system (\ref{a1})-(\ref{ab}) implies 
that either $b=Ax$ for some integral vector $x\in\nn^n$ or there exists 
$\xi=(\xi_w)\in\Omega$ such that $p_\xi(b)<0$ for a polynomial 
$p_\xi\in\rr[u_1,...,u_m]$ of the form
\begin{equation}\label{different}
u\mapsto p_\xi(u)\,=\,\sum_{w\in\nn^m,\,w\neq0}\xi_w\,u^w.
\end{equation}
Thus (\ref{different}) provides an explicit nonlinear {\it polynomial} 
certificate for IP, in contrast with the {\it linear} 
polynomial certificate fro LP obtained from  the classical Farkas lemma.

In the discrete Farkas lemma presented in 
\cite{lasserrefarkas,lasserreduality} the author defines a
polyhedral cone $\Omega_2\subset\rr^s$ associated 
with (\ref{laas}), and proves that either $b=Ax$ for 
some nonnegative integral vector $x\in\nn^n$ or there exists 
$\pi\in\Omega_2$ such that $(-1,0,...,0,1)\cdot\pi<0$. It turns out 
that the vector $\xi=(\xi_w)$ in (\ref{different}) indeed satisfies 
$\pi=\Delta'\xi$ for a square invertible matrix $\Delta$ defined
as the Kronecker product of Vandermonde matrices. 

\vspace{1em} $\bullet$ 
Inspired by the relationships between the systems (\ref{a1})-(\ref{ab}) 
and (\ref{laas}), we finally show that existence of a nonnegative 
integral solution $x\in\nn^n$ for the linear system $Ax=b$ reduces 
to the existence of a nonnegative real solution $\y\in\rr^p$ for a 
system of linear equations of the form:
$$M\,(-1,0,...,0,1)'\,=\,M\,\Theta\,\y,$$
where $M\in\rr^{s\times{s}}$ is any square invertible matrix. We can 
apply the standard Farkas lemma to any one of the linear systems presented 
above, and deduce a specific (Farkas) certificate for each 
choice the invertible matrix $M$. Each certificate can be seen 
as a \textit{theorem of the alternative} of the form: the system $b=Ax$ has no 
nonnegative integral solution $x\in\nn^n$ if and only if $f(b)<0$ for 
some given function $f$. For the particular choice
$M:=\Delta'$, existence of a nonnegative integral solution 
$x\in\nn^n$ for the linear system $Ax=b$ reduces to existence of a 
nonnegative real solution $\y\in\rr^p$ for a system of linear equations 
of the form: 
$$u^b-1\,=\,\y'D^{[u]},\quad\forall\;u\in\nn^m
\quad\hbox{with}\quad{u}\leq b,$$ 
for some appropriate nonnegative integer vectors $D^{[u]}\in\nn^{p}$. 
In this case, the function $f$ involved in the Farkas certificate 
$f(b)<0$ has the exponential-like expansion
\begin{equation}\label{expo}
u\mapsto\,f(u)\,:=\,\sum_{z\in\nn^m}\xi_z\cdot(z^u-1).
\end{equation}

Both certificates (\ref{different}) and (\ref{expo}) are different from the 
certificate 
obtained from the superadditive dual approach 
of Gomory and Johnson \cite{gomoryjohnson}, Johnson \cite{johnson}, and 
Wolsey \cite{wolsey1}. In particular, in \cite{wolsey1} the
linear system associated with such a certificate has dimension $s^2\times s$,
which is larger than the size $ns\times{s}$ of the system 
(\ref{laas}) of this paper.

\vspace{1em}\subsection*{The method} 
A theorem of the alternative is obtained in three steps:

 \textbf{\{1\}} One first shows that a linear system $Ax=b$ (with $A\in\nn^{m\times n}$ 
and $b\in\nn^m$) has a nonnegative integral solution $x\in\nn^n$ if and 
only if the function $f:\zz^m{\to}\rr$ given by $z\mapsto{f}(z):=b^z$ 
can be written as a linear combination of functions
$z\mapsto{f_u}(z):=(u{+}A_k)^z{-}u^z$ weighted by some nonnegative 
coefficients (for the indexes $u\in\nn^m$ with $u\leq b$).

 \textbf{\{2\}} 
One then shows that computing the nonnegative coefficients in the above 
linear decomposition of $f$  is equivalent to finding a nonnegative real 
solution to a finite system of linear equations whose dimension is bounded 
from above by  $ns\times{s}$ with $s:=\prod_j(1{+}b_j)$.

 \textbf{\{3\}} 
Finally one applies the standard continuous Farkas lemma to the linear 
system described in  \textbf{\{2\}} and obtains the certificate 
(\ref{different}).

This approach is similar in flavor but different from the one in 
\cite{lasserrefarkas,lasserreduality}.
However, they are strictly equivalent and can be 
deduced one from each other under some appropriate linear transformation 
whose associated matrix $\Delta\in\nn^{s\times s}$ has a simple explicit 
form. 

\section{Notation and definitions}

The notation $\rr$, $\zz$ and $\nn=\{0,1,2,...\}$ stand for the usual 
sets of real, integer and natural numbers, respectively. Moreover, the 
set of positive integer numbers is denoted by $\nn^*=\nn\setminus\{0\}$. 
Given any vector $b\in\zz^m$ and matrix $A\in\zz^{m\times{n}}$, the 
$[k]$-entry of $b$ (resp. $[j;k]$-element of $A$) is denoted by either
$b_k$ or $b[k]$ (resp. $A_{j;k}$ or $A[j;k]$). The notation $A'$ stands 
for the transpose of any matrix (or vector) $A\in\rr^{m\times{n}}$; and the 
$k$th column of the matrix $A$ is denoted by $A_k:=(A_{1;k},...,A_{m;k})'$.

\vspace{1em}\subsection*{The semi-group membership problem}
A classical result in group theory states that every abelian group $G$ 
with $m\in\nn^*$ generators is isomorphic to some Cartesian product
$$G\,\cong\,[\zz/p_1\zz]\times[\zz/p_2\zz]\times
\cdots\times[\zz/p_q\zz]\times[\zz^{m-q}],$$
for some set of numbers $\{p_j\}\subset \nn^*$ and $q\leq{m}$; see 
e.g. \cite{DuFo}. One may even suppose that every $p_j$ divides $p_k$ 
whenever $j<k$. Therefore, if one introduces the \textit{extended} 
$m$-dimensional vector:
\begin{equation}
P\,:=\,(p_1,p_2,\ldots,p_q,\infty,\ldots,\infty)',
\end{equation}
the abelian group $G$ is isomorphic to the group $\widetilde{G}$ 
of vectors $x\in\zz^m$ such that $0\leq{x_j}<p_j$ 
for every $1\leq{j}\leq{q}$; notice that $q\leq{m}$. The group sum 
$x\oplus y$ of two elements $x$ and $y$ in $\widetilde{G}\subset\zz^m$ 
is then defined by 
\begin{equation}
x\oplus{y}:=(x+y)\bmod{P}\quad\mbox{in}\quad\widetilde{G},
\end{equation}
where the sum $x+y$ is the standard addition on $\zz^m$ and the 
modulus of the sum $(x+y)\bmod{P}$ is calculated entry by entry, 
so that for every index $1\leq{k}\leq{m}$, 
\begin{equation}
\big[(x+y)\bmod{P}\big]_k\,=\,\left\{\begin{array}{ccc}
(x_k+y_k)\bmod{P_k}&\mbox{if}&P_k<\infty,\\
(x_k+y_k)&\mbox{if}&P_k=\infty.
\end{array}\right.\end{equation}

Hence from now on we suppose that $G=\widetilde{G}\subset\zz^m$. Next 
let $\{a_k\}\subset{G}$ be a collection of $n$ elements of $G$. Each 
element $a_k$ can be seen as a vector of $\zz^m$, for $1\leq{k}\leq{n}$,  
so that the  semi-group  generated by $\{a_k\}$ is the same as the set
\begin{eqnarray*}
G_a&:=&\{Ax\bmod{P}\,|\,x\in\nn^m\},\quad\hbox{with}\\
A&:=&[a_1|a_2|\cdots|a_n]\in\zz^{m\times{n}}.
\end{eqnarray*}

Thus, given $b\in G$, the semi-group membership problem of deciding
whether $b\in{G}_a$  is equivalent to deciding whether the system of
linear equations $b=Ax\bmod{P}$ has a solution $x\in\nn^n$. This in 
turn is equivalent to deciding whether the following system of linear 
equations 
\begin{equation}
b\,=\,Ax+\begin{pmatrix}-B&B\\ 0&0\end{pmatrix}
\cdot\begin{pmatrix}u\\ w\end{pmatrix},
\end{equation}
has a solution $(x,u,w)$ in $\nn^n\times\nn^q\times\nn^q$,
where
\begin{equation}
B\,:=\,\begin{pmatrix}
p_1&0&\cdots&0\\  
0&p_2&\cdots&0\\ 
\vdots&\vdots&\cdot&\vdots\\  
0&0&\cdots&p_q\\
\end{pmatrix}\,\in\,\nn^{q\times{q}}.
\end{equation}

Hence, with no loss of generality, the membership problem is equivalent 
to deciding whether some related  system of linear equations $\mathcal{A}x=b$ 
(with $\mathcal{A}\in\zz^{m\times\ell}$ and $b\in\zz^m$) has a 
solution $x\in\nn^\ell$, which is the problem we will consider 
in the sequel.

\section{Existence of integer solutions to $\{Ax=b,\,x\geq0\}$}

Let $\rr[z]=\rr[z_1,\ldots,z_m]$ be the ring of real polynomials 
in the variables $z=(z_1,\ldots,z_m)$ of $\rr^m$. With 
$\mathcal{A}\in\zz^{m\times\ell}$ we analyze existence of a 
nonnegative integer solution $x\in\nn^\ell$ to the system of 
linear equations $\mathcal{A}x=b$.

According to the computational complexity terminology in \cite{schrijver}, 
let $\varphi$ be the facet complexity of the rational convex polyhedron 
$\p:=\{x\in\rr^\ell:x\geq0,\mathcal{A}x=b\}$. That is, each inequality 
in $\mathcal{A}x\leq b$, $\mathcal{A}x\geq b$, and $x\geq0$ has size at 
most $\varphi$. Corollary~17.1b in \cite[p.~239]{schrijver} states that
$\p$ contains an integral vector of size at most $6\ell^3\varphi$, if 
$\p$ contains an integral vector $x\in\nn^\ell$. Hence existence of 
an integral vector in $\p$ is equivalent to analyze existence of an 
integral vector of the convex (compact) polytope 
$$\widehat{\p}\,:=\,\{x\in\rr^\ell:x\geq0,\,\mathcal{A}
x=b,\,\sum_{i=1}^\ell x_i\leq M_{\mathcal{A},b}\},$$ 
where $M_{A,b}$ is obtained explicitly from the facet complexity of $\p$.

\vspace{1em}\subsection{Reduction to $A\in\nn^{m\times n}$}
In view of the above one may restrict our analysis to the existence 
of a nonnegative integral solution $x\in\nn^n$ for a system of linear 
equations $Ax=b$ associated with a rational convex {\it polytope} 
$\p:=\{x\in\rr^n:Ax=b,x\geq0\}$, where $A\in\zz^{m\times{n}}$ and 
$b\in\zz^m$. But this in turn implies that one may restrict our 
analysis to existence of a nonnegative integral solution 
$y\in\nn^{n+1}$ to a system of linear equations $A^{\star}y=b^\star$ 
where $A^\star\in\nn^{[m+1]\times[n+1]}$ and $b^\star\in\nn^{m+1}$.

Indeed, if  $A\in\zz^{m\times{n}}$ and $A\not\in\nn^{m\times{n}}$,
let $\alpha\in\nn^n$ be such that
\begin{equation}
\label{3-10}
\widehat{A}_{j;k}:=A_{j;k}+\alpha_k\geq0;\quad
\forall\;1\leq{j}\leq{m},\;1\leq{k}\leq{n}.
\end{equation}
Since $\p$ is a (compact) polytope,
\begin{equation}\label{3-11}
\rho\,:=\,\max_{x\in\nn^n,\,Ax=b}\bigg\{
\sum_{k=1}^n\alpha_kx_k\bigg\}\,<\,\infty.
\end{equation}
In particular $\rho\in\nn$. Now define $\widehat{b}\in\nn^m$ by 
\begin{equation}\label{3-12}
\widehat{b}:=b+\rho\,\e_m\geq0\quad
\hbox{with}\quad\e_m:=(1,...,1)'\in\nn^m.
\end{equation}

Let $\widehat{A}\in\nn^{m\times n}$ be defined as in (\ref{3-10}). 
The solutions $x\in\nn^n$ to the original system $Ax=b$ are in 
one-to-one correspondence with the solutions $(x,u)\in\nn^n\times\nn$ 
to the extended system
\begin{equation}\label{3-20}
\begin{array}{rl}
\widehat{A}x+\e_mu&=\widehat{b},\\
\alpha'x+u&=\rho.
\end{array}\end{equation}
Indeed, if $Ax=b$ with $x\in\nn^n$, then
$$Ax+\e_m\bigg[\sum_{k=1}^n\alpha_kx_k\bigg]
-\e_m(\alpha'x)+\rho\,\e_m\,=\,b+\rho\,\e_m.$$
The following identity follows from the definitions for 
$\widehat{A}\in\nn^{m\times n}$ and $\widehat{b}\in\nn^m$ 
given in (\ref{3-10})-(\ref{3-12}) :
$$\widehat{A}x+\e_mu=\widehat{b}\quad
\hbox{with}\quad{u}:=\rho-\alpha'x\in\zz.$$

Notice that $u\geq0$ because $\rho\geq\alpha'x$ according 
to (\ref{3-11}), so that $(x,u)$ is the integer nonnegative 
solution of (\ref{3-20}) that we are looking for. Conversely let 
$(x,u)\in\zz^{n+1}$ be a solution to (\ref{3-20}). The definitions 
for $\widehat{A}\in\nn^{m\times n}$ and $\widehat{b}\in\nn^m$ 
given in (\ref{3-10})-(\ref{3-12}) imply that 
$$Ax+\e_m\bigg[\sum_{k=1}^n\alpha_kx_k\bigg]+\e_mu=b+\rho\,\e_m,$$
so that $Ax=b$ with $x\in\nn^n$ because $u=\rho{-}\alpha'x$. Hence 
the existence any solution $x\in\nn^n$ to $Ax=b$ is completely 
equivalent to the existence any solution $(x,u)\in\nn^{n+1}$ to 
$$\begin{pmatrix}\;\widehat{b}\;\\ \rho\end{pmatrix}=A^\star\cdot
\begin{pmatrix}x\\ u\end{pmatrix}\quad\hbox{with}\quad{A}^\star
:=\begin{pmatrix}\widehat{A}&\e_m\\ \alpha'&1\end{pmatrix},$$
and the new matrix $A^\star\in\nn^{[m+1]\times[n+1]}$ 
has only nonnegative integer entries.

\vspace{1em}\subsection{The main result}
Given vectors $b\in\nn^m$ and $z\in\rr^m$, the notation $z^b$ stands 
for the monomial $z_1^{b_1}z_2^{b_2}\cdots{}z_m^{b_m}\in\rr[z]$. We 
also need to define a pair of matrices $\Delta$ and $\Theta$ that we 
will use in the sequel.

\begin{definition}\label{main-def}
Let $A\in\nn^{m\times{n}}$ and $\beta\in\nn^m$ be such that 
$A_k\leq\beta$ for each index $1\leq{k}\leq{n}$. Set the integers
\begin{equation}\label{main-1}
s:=\prod_{j=1}^m(1{+}\beta_j)\quad\hbox{and}\quad
p:=\sum_{k=1}^n\prod_{j=1}^m\big(1{+\beta_j-}A_{j;k}\big)\leq{ns}.
\end{equation}

Let $\Delta\in\nn^{s\times{s}}$ be a square Vandermonde matrix 
whose rows and columns are indexed with the nonnegative vectors 
$z,w\in\nn^m$ (according to e.g. the lexicographic ordering) 
so that the $[z;w]$-entry is given by
\begin{equation}\label{matrix-2}
\Delta[z;w]=w^z=w_1^{z_1}w_2^{z_2}\cdots
w_m^{z_m}\quad\forall\;z,w\leq\beta;
\end{equation}
and where we use the convention that $0^0=1$. 

Let $\Theta\in\zz^{s\times{p}}$ be a network matrix whose rows 
and columns are respectively indexed with the nonnegative vectors 
$w\in\nn^m$ and $(k,u)\in\nn{\times}\nn^m$, so that the 
$(w;(k,u))$-entry is given by
\begin{equation}\label{matrix-3}
\Theta[w;(k,u)]\,=\,\left\{\begin{array}{cl} 
-1&\hbox{if~}w=u\leq\beta{-}A_k,\\ 
1&\hbox{if~}w=u{+}A_k\leq\beta,\\
0&\hbox{otherwise},\end{array}\right.
\end{equation}
for all indexes $w\leq\beta$, $u\leq\beta{-}A_k$, and $1\leq{k}\leq{n}$.
\end{definition}

The following result is straightforward.

\begin{lemma}\label{Kroneker}
The square Vandermonde matrix $\Delta\in\rr^{s\times{s}}$ 
in (\ref{matrix-2}) is invertible. The matrix 
$\Theta\in\rr^{s\times{p}}$ in (\ref{matrix-3}) is 
a network matrix, and so it is totally unimodular.
\end{lemma}

\begin{proof}
It is easy to see that $\Delta$ is invertible because it is the 
Kroneker product $D^{[1]}\otimes{D}^{[2]}\otimes\cdots\otimes{D^{[m]}}$ 
of $m$ square Vandermonde matrices
\begin{equation}
D^{[j]}\,=\,\begin{pmatrix}1&1^0&2^0&\cdots&(\beta_j)^0\\  
0&1^1&2^1&\cdots&(\beta_j)^1\\ \vdots&\vdots&\vdots&\cdot&\vdots\\ 
0&1^{\beta_j}&2^{\beta_j}&\cdots&(\beta_j)^{\beta_j}\end{pmatrix};
\end{equation}
see e.g. \cite{hornjohnson,zhang}. And each $D^{[j]}$ is 
obviously invertible. By inspection it turns out that $\Theta$ is a 
\textit{network matrix}, that is, it is a matrix with only $\{0,\pm1\}$ 
entries and with exactly two nonzero entries $1$ and $-1$ in each 
column. Therefore $\Theta$ is totally unimodular; see e.g. Schrijver 
\cite[p.~274]{schrijver}. 
\end{proof}

\begin{definition}
Given a finite set $U\subset\nn^m$ and a collection of 
coefficients $q[u]\in\rr$, we say that the polynomial
$$z\to{Q}(z)\,=\,\sum_{u\in{U}}q[u]\,z^u\,=\,
\sum_{u\in{U}}q[u]z_1^{u_1}z_2^{u_2}\cdots{z_m}^{u_m}$$ 
has multivariate degree bounded by a vector $\beta\in\nn^m$ 
if and only if $u\leq\beta$ for every $u\in{U}$.
\end{definition}

The following technical result was essentially shown
in \cite{lasserrefarkas,lasserreduality} but we include 
the proof for the sake of completeness.

\begin{theorem}\label{Lasserre}
Let $\beta,b\in\nn^m$ and $A\in\nn^{m\times{n}}$ be such that~:
\begin{equation}\label{eqn-p1}
\beta\geq{b}\quad\hbox{and}\quad\beta\geq
A_k\quad\hbox{for all}\quad1\leq{k}\leq{n}.
\end{equation}
The following three statements {\rm\bf (a)}, 
{\rm\bf (b)} and {\rm\bf (c)} are all equivalent~:
\begin{description}
\item[(a)] The linear system $Ax=b$ has a solution $x\in\nn^n$.

\vspace{1em}\item[(b)] The polynomial 
$z\mapsto{z}^b{-}1:=z_1^{b_1}z_2^{b_2}\cdots{z}_m^{b_m}{-}1$ 
can be written as follows~:
\begin{equation}\label{eqn-p2}
z^b-1\,=\,\sum_{k=1}^nQ_k(z)(z^{A_k}-1)
\end{equation}
for some real polynomials $Q_k\in\rr[z_1,z_2,...,z_m]$ with 
nonnegative coefficients and multivariate degree bounded 
by the vector $\beta{-}A_k$, for all $1\leq{k}\leq{n}$.

\vspace{1em}\item[(c)] There is a real nonnegative solution 
$\y\in\rr^p$ for the system of linear equations
\begin{equation}\label{eqn-p3}
\bb\,=\,\Theta\,\y,
\end{equation}
where $\Theta\in\zz^{s\times{p}}$ is given as in (\ref{matrix-3}) 
of Definition~\ref{main-def} and the new vector $\bb\in\zz^s$ has 
entries indexed with the nonnegative vector $w\in\nn^m$, so that 
the $[w]$-entry is given by
\begin{equation}\label{eqn+p3}
\bb[w]\,:=\,\left\{\begin{array}{cl}
-1&\hbox{if~}w=0,\\ 1&\hbox{if~}w=b,\\
0&\hbox{otherwise.}\end{array}\right.
\quad\forall\;0\leq{w}\leq\beta.
\end{equation}
\end{description}
\end{theorem}

Notice that $\Theta\in\zz^{s\times{p}}$ in (\ref{matrix-3}) and 
(\ref{eqn-p3}) is determined only by the entries of $\beta\in\nn^m$ 
and  $A\in\nn^{m\times{n}}$, so that $\Theta$ is independent of 
$b\in\zz^m$. 
Observe that if $b\geq{A_k}$
for every $1\leq{k}\leq{n}$, then one may take $\beta:=b$.

\begin{proof}
\textbf{(a)} $\Rightarrow$ \textbf{(b)}. 
Suppose that $b=A x$ for some $x\in\nn^n$. Consider 
the following polynomials (where we use the convention 
$\sum_{q=0}^{-1}(\cdot)=0$)
\begin{eqnarray*}
z\mapsto Q_1(z)&=&\sum_{q=0}^{x_1-1}z^{qA_1},\\
z\mapsto Q_2(z)&=&z^{x_1A_1}\sum_{q=0}^{x_2-1}z^{qA_2},\\
z\mapsto Q_3(z)&=&z^{x_1A_1}z^{x_2A_2}\sum_{q=0}^{x_3-1}z^{qA_3},\\
&\cdots&\\ z\mapsto Q_n(z)&=&\bigg[\prod_{k=1}^{n-1}
z^{x_kA_k}\bigg]\;\sum_{q=0}^{x_n-1}z^{qA_n}.\\
\end{eqnarray*}
 
It is easy to see that the polynomials $Q_k$ satisfy equation 
(\ref{eqn-p2}). Moreover each $Q_k$ has nonnegative coefficients, 
and the multivariate degree of $Q_k$ is bounded by the vector 
$\beta{-}A_k$ for $1\leq{k}\leq{n}$ because $b\leq\beta$.

\vspace{1em}\textbf{(b)} $\Leftrightarrow$ \textbf{(c)}. 
Existence of polynomials $Q_k\in\rr[z_1,z_2,...,z_m]$ with 
nonnegative coefficients and multivariate degree bounded by the 
vector $\beta{-}A_k$ (for all $1\leq{k}\leq{n}$) is equivalent 
to the existence of real coefficients $y[k,u]\geq0$ such that
$$Q_k(z)=\!\sum_{u\in\nn^m\!,\,u\leq\beta-A_k}\!
y[k,u]z^u\quad\forall\quad1\leq{k}\leq{n}.$$
Then rewrite equation (\ref{eqn-p2}) as follows:
\begin{equation}\label{eqn-p4}
z^b-1=\sum_{k=1}^n\;
\sum_{\genfrac..{0pt}1{u\in\nn^m,}{u\leq\beta-A_k}}
y[k,u]\big(z^{u+A_k}-z^u\big),
\end{equation}
The vector of coefficients $\y=(y[k,u])\geq0$ satisfies 
(\ref{eqn-p4}) if and only if it satisfies the following 
system of linear equations
\begin{equation}\label{eqn-p5}
\bb\,=\,\Theta\cdot\begin{pmatrix}y[1,(0,...,0,0)]\\ 
y[1,(0,...,0,1)]\\ \vdots\\ y[n,(\beta-A_n)]\end{pmatrix},
\end{equation}
where $\bb\in\zz^s$ is given in (\ref{eqn+p3}) and the matrix 
$\Theta\in\zz^{s\times p}$ is given in (\ref{matrix-3}), that 
is, such that
$$\Theta[w;(k,u)]\,=\,\left\{\begin{array}{cl}
-1&\hbox{if~}w=u\leq{b-}A_k,\\ 
1&\hbox{if~}w=u{+}A_k\leq{b},\\
0&\hbox{otherwise.}\end{array}\right.$$

Each row in (\ref{eqn-p5}) is indexed with the monomial 
$z^w\in\nn^m$, for $w\in\nn^m$ and $w\leq{b}$ (according 
to e.g. the lexicographic ordering).

\vspace{1em}\textbf{(c)} $\Rightarrow$ \textbf{(a)}. 
Let $\y=(y[k,u])\geq0$ be a real vector such that (\ref{eqn-p3}) 
and (\ref{eqn-p5}) hold. The matrix $\Theta$ is totally unimodular
according to Lemma~\ref{Kroneker}, and so there exists a nonnegative 
\textit{integer} solution $\widehat{\y}=(\widehat{y}[k,u])\geq0$ to 
(\ref{eqn-p3}) and (\ref{eqn-p5}) because the left-hand-side of 
(\ref{eqn-p5}) is an integral vector. Therefore (\ref{eqn-p4}) 
also holds with the new vector $\widehat{\y}$, that is:
$$z^b-1=\sum_{k=1}^n\;
\sum_{\genfrac..{0pt}1{u\in\nn^m,}{u\leq\beta-A_k}}
\widehat{y}[k,u]\big(z^{u+A_k}-z^u\big).$$
Given a fixed index $1\leq{j}\leq{m}$, differentiate both sides 
of the equation above with respect to the variable $z_j$ and 
evaluate at the point $z=(1,1,...,1)$, so as to obtain:
\begin{equation}\label{eqn-p6}
b_j=\sum_{k=1}^nA_{j;k}x_k,\quad\hbox{with}\quad
x_k:=\sum_{\genfrac..{0pt}1{u\in\nn^m,}{u\leq\beta-A_k}}
\widehat{y}[k,u].
\end{equation}
The nonnegative integer vector $x=(x_1,...,x_n)'\in\nn^m$ 
satisfies the desired result $Ax=b$.
\end{proof}

\begin{rem}\label{rem33}
One linear constraint in (\ref{eqn-p3}) is redundant, because the 
addition of all the rows in (\ref{eqn-p3})-(\ref{eqn+p3}) yields 
the trivial equality $0=0\,\y$ (recall that $\Theta$ is a matrix 
with only $\{0,\pm1\}$ entries and with exactly two nonzero 
entries $1$ and $-1$ in each column).
\end{rem}

\begin{rem}\label{rem44}
The matrix $\Theta\in\zz^{s\times{p}}$ in (\ref{matrix-3}) is 
independent of $b$ and contains all information about all 
the integer problems $Ax=b$ for $x\in\nn^n$ and $b\in\nn^m$ 
with $b\leq\beta$.
\end{rem}

The linear system $\bb=\Theta\,\y$ in (\ref{eqn-p3}) of 
Theorem~\ref{Lasserre} is quite interesting in the sense that 
existence of a real solution $\y\geq0$ is totally equivalent to 
the existence of a nonnegative integer solution $x\in\nn^n$ for 
the problem $Ax=b$, but it is not easy to see at first glance what 
is the relation between the two solutions $\y$ and $x$. Moreover 
equation (\ref{eqn-p3}) in Theorem~\ref{Lasserre} is not unique at 
all. We can multiply both sides of (\ref{eqn-p3}) by any square 
invertible matrix $M\in\rr^{s\times{s}}$ and we obtain that 
$\y\geq0$ is a real solution to (\ref{eqn-p3}) if and only 
if $\y$ is also a solution for
$$M\,\bb\,=\,M\,\Theta\,\y.$$

An interesting issue is to determine what is an adequate matrix 
$M$ such that relationship between the solutions $\y\geq0$ 
and $x\in\nn^n$ is evident. We claim that one of such matrix $M$ 
is the square Vandermonde matrix $\Delta\in\nn^{s\times{s}}$ presented 
in (\ref{matrix-2}) of Definition~\ref{main-def}. We use the 
convention that $0^0=1$ and $0^z=0$ whenever $0\neq z\in\nn^m$.

\begin{theorem}\label{polynomial}
Let $\beta,b\in\nn^m$ and $A\in\nn^{m\times{n}}$ be such that~:
\begin{equation}\label{poly-1}
\beta\geq{b}\quad\hbox{and}\quad\beta\geq
A_k\quad\hbox{for all}\quad1\leq{k}\leq{n}.
\end{equation}
The following three statements {\rm\bf (a)}, 
{\rm\bf (b)} and {\rm\bf (c)} are all equivalent~:
\begin{description}
\item[(a)] The linear system $Ax=b$ has a solution $x\in\nn^n$.

\vspace{1em}\item[(b)] There is a real nonnegative solution 
$\y=(y[k,u])\in\rr^p$ to the system of linear equations
\begin{equation}\label{poly-2}
b^z-0^z=\sum_{k=1}^n\;
\sum_{\genfrac..{0pt}1{u\in\nn^m,}{u\leq\beta-A_k}}
y[k,u]\big((u{+}A_k)^z-u^z\big),
\end{equation}
for every $z\in\nn^m$ with $0\leq{z}\leq\beta$.

\vspace{1em}\item[(c)] There is a real nonnegative solution 
$\y=(y[k,u])\in\rr^p$ to the system of linear equations
\begin{equation}\label{poly-3}
z^b-1=\sum_{k=1}^n\;
\sum_{\genfrac..{0pt}1{u\in\nn^m,}{u\leq\beta-A_k}}
y[k,u]\big(z^{u+A_k}-z^u\big),
\end{equation}
for every $z\in\nn^m$ with $0\leq{z}\leq\beta$.
\end{description}
\end{theorem}


\begin{proof} The equivalence \textbf{(a)} $\Leftrightarrow$ 
\textbf{(b)} easily follows from Theorem~\ref{Lasserre} after 
noticing that the $[z;(k,u)]$-entry (resp. $[z]$-entry) of the 
product $\Delta\Theta$ (resp. $\Delta\bb$) is given by 
\begin{eqnarray}\label{poly-4}
(\Delta\Theta)[z;(k,u)]&=&\big((u{+}A_k)^z-u^z\big)\in\nn,\\
\label{poly-5}\hbox{resp.}\quad(\Delta\bb)[z]&=&\big(b^z-0^z\big),
\end{eqnarray}
for all nonnegative indexes $z\in\nn^m$ and $(k,u)\in\nn{\times}\nn^m$ 
with $z\leq\beta$, $u\leq\beta{-}A_k$ and $1\leq{k}\leq{n}$. 

The equivalence \textbf{(a)} $\Leftrightarrow$ \textbf{(c)} is proved 
in a similar way; we only need to observe that multiplying by the 
transpose $\Delta'$ is equivalent to interchange the exponentials 
$z\to{w^z}$ by powers $z\to{z^w}$, so that:
\begin{eqnarray*}
(\Delta'\,\Theta)[z;(k,u)]&=&\big(z^{u+A_k}-z^u\big)\in\nn,\\
\hbox{and}\quad(\Delta'\,\bb)[z]&=&\big(z^b-z^0\big).
\end{eqnarray*}
\end{proof}

\begin{rem}\label{rem55}
A nonnegative real vector $\y\in\rr^p$ is a solution of equation 
(\ref{eqn-p3}) in Theorem~\ref{Lasserre} if and only if $\y$ is also 
a solution of (\ref{poly-2}) and (\ref{poly-3}). Moreover the linear 
system (\ref{poly-3}) can be directly deduced from (\ref{eqn-p2}) 
by evaluating at $z\in\nn^m$ with $z\leq{b}$.
\end{rem} 

\section{A hierarchy of linear programming relaxations}
Let $\beta,b\in\nn^m$ and $A\in\nn^{m\times{n}}$ be such that 
$\beta\geq{b}$ and $\beta\geq{A_k}$ for all $1\leq{k}\leq{n}$. Consider 
the matrices $\Delta\in\rr^{s\times{s}}$ and $\Theta\in\rr^{s\times{p}}$ 
given in equations (\ref{matrix-2})-(\ref{matrix-3}) of 
Definition~\ref{main-def}; and define the new vector $\bb\in\zz^s$ 
whose entries are indexed with the non negative vector $w\in\nn^m$, 
so that the $[k]$-entry is given by (\ref{eqn+p3})~:
\begin{equation}\label{main-vector}
\bb[w]\,:=\,\left\{\begin{array}{cl}
-1&\hbox{if~}w=0,\\ 1&\hbox{if~}w=b,\\ 
0&\hbox{otherwise,}\end{array}\right.
\quad\forall\;0\leq{w}\leq\beta.
\end{equation}

The fact that $\Delta$ is invertible (according to Lemma~\ref{Kroneker}) 
allows us to give two equivalent definitions of the following polytope 
in $\rr^p$,
\begin{eqnarray}\label{main-polytope}
\mathcal{P}&:=&\{\y\in\rr^p:\y\geq0,
\,\Delta\Theta\y=\Delta\bb\}\\
\nonumber&=&\{\y\in\rr^p:\y\geq0,
\,\Theta\y=\bb\}\,\subset\,\rr^p.
\end{eqnarray}

As it is state in Remark~\ref{rem55}, any nonnegative vector $\y\in\rr^p$ 
is a solution to (\ref{poly-2}) if and only if $\y$ lies in $\mathcal{P}$. 
Let $\y=(y[k,u])$ be an element of $\mathcal{P}$. 
The entries of $\y$ are indexed according to 
Theorems~\ref{Lasserre} and~\ref{polynomial}. 
Equation~(\ref{poly-2}) implies that $x\in\rr^n$ with $x_k:=\sum_uy[k,u]$, 
is a solution to $Ax=b$. Indeed $b_j=\sum_kA_{j;k}x_k$ for 
$1\leq{j}\leq{m}$ after evaluating (\ref{poly-2}) at $z=e_j$, the 
basic vector whose entries are all equal to zero, except
the $j$-entry which is equal to one, so that $w^{e_j}=w_j$ for all 
$w\in\zz^m$. The relationship between $x$ and $\y$ 
comes from a multiplication by some matrix $E\in\nn^{n\times{p}}$, that is~:
\begin{equation}\label{matE}
x=E\y\quad\mbox{with}\quad{E}:=\hspace{-1ex}
\bordermatrix{&p_1&p_2&&p_n\cr&\overbrace{1\ldots1}
&\overbrace{0\ldots0}&\cdots&\overbrace{0\ldots0}\cr
&0\ldots0&1\ldots1&\cdots&0\ldots0\cr
&\cdots&\cdots&\cdot&\cdots\cr
&0\ldots0&0\ldots0&\cdots&1\ldots1\cr},
\end{equation}
where $p_k:=\prod_{j=1}^m\big(1{+\beta_j-}A_{j;k}\big)$ for 
$1\leq{k}\leq{n}$, so that $p=\sum_kp_k$ according to~(\ref{main-1}) 
in Definition~\ref{main-def}. We obtain the following result as a 
consequence of Theorem \ref{polynomial}.

\begin{corollary}\label{main}
Let $\beta,b\in\nn^m$ and $A\in\nn^{m\times{n}}$ be such that 
$\beta\geq{b}$ and $\beta\geq{A_k}$ for each index $1\leq{k}\leq{n}$. 

{\bf (a):} Existence of a nonnegative integral solution $x\in\nn^n$ to 
the linear system $Ax=b$ reduces to the existence of a nonnegative real 
solution $\y=(y[k,u])\in\rr^p$ to the system of linear equations~:
\begin{eqnarray}
\label{m-a1}
b_i&=&\y'A^{(e_j)},\quad\forall\;1\leq{j}\leq{m};\\
\label{m-a2}
b_ib_j&=&\y'A^{(e_i+e_j)},\quad\forall\;1\leq{i}\leq{j}\leq{m};\\
\nonumber\dots&=&\cdots\\
\label{m-am}
b_1^{z_1}\cdots{}b_m^{z_m}&=&\y'A^{(z)},\quad
\begin{pmatrix}\forall\;0\leq{z_j}\leq{b_j},\\
z_1{+}\cdots{+}z_m=\delta;\end{pmatrix}\\
\nonumber\dots&=&\cdots\\
\label{m-ab}
b_1^{\beta_1}\cdots{}b_m^{\beta_m}&=&\y'A^{(\beta)};
\end{eqnarray}
for some appropriate nonnegative integer vectors 
$A^{(e_j)},A^{(e_i+e_j)},A^{(z)}\in\nn^p$ with $z\in\nn^m$ and $z\leq\beta$. 
The parameter $\delta\geq1$ in (\ref{m-am}) is the degree of the 
monomial $b\mapsto{b^z}$, and we use the indexation for the entries 
of $\y$ given in Theorems~\ref{Lasserre} and~\ref{polynomial}. In 
particular all vectors $A^{(z)}$ are independent of $b$, and the 
entries of $A^{(e_j)}\in\nn^{p}$ in (\ref{m-a1}) are given by~: 
$$A^{(e_j)}[k,u]\,:=\,A_{j;k},$$
for all $u\in\nn^m$, $1\leq{j}\leq{m}$, and 
$1\leq{k}\leq{n}$ with $u\leq\beta{-}A_k$.

\bigskip
{\bf (b):} Let $\y\in\rr^p$ be a nonnegative real solution 
to (\ref{m-a1})-(\ref{m-ab}) and $\widehat{x}=E\y$, i.e. 
$\widehat{x}_k:=\sum_uy[k,u]$ for $1\leq{k}\leq{n}$. Then $\widehat{x}$ 
belongs to the integer hull $H$ of $\{x\in\rr^n:Ax=b,x\geq0\}$. Therefore 
a certificate of $b\neq{Ax}$ for any $x\in\nn^n$ is obtained as soon as 
any subsystem of (\ref{m-a1})-(\ref{m-ab}) has no solution $\y\in\rr^p$. 

Moreover, any nonnegative vector $\y\in\rr^p$ is a solution to
(\ref{m-a1})-(\ref{m-ab}) if and only if $\y\in\mathcal{P}$ (with
$\mathcal{P}$ as in (\ref{main-polytope})); and there is a one-to-one correspondance
between the vertices of $H$ and $\mathcal{P}$.
\end{corollary}


\begin{proof}
{\bf (a)} It is straightforward to see that (\ref{m-a1})-(\ref{m-ab}) is just 
an explicit description of (\ref{poly-2}) in Theorem~\ref{polynomial}. 
The right-hand-side terms of (\ref{m-a1})-(\ref{m-ab}) are the function 
$b^z{-}0^z$ evaluated at $z\in\nn^m$ with $z\leq\beta$ and $\sum_jz_j$ 
equal to $1$, $2$, $3$, etc., until the maximal degree $\sum_j\beta_j$. 
Each vector $A^{(z)}\in\nn^p$ is then the transpose of the $[z]$-row 
of the product $\Delta\Theta$ calculated in (\ref{poly-4}), so that 
$$A^{(z)}[k,u]:=(\Delta\Theta)[z;(k,u)]=\big((u{+}A_k)^z-u^z\big)\in\nn,$$ 
for all appropriated indices $z$, $k$, and $u$. Evaluating (\ref{poly-2}) 
at $z=0$ yields the trivial identity $0=0\y$ (recall Remark~\ref{rem33}), 
because we are using the convention that $0^0=1$ and $0^w=0$ for every 
$w\neq0$. While evaluating (\ref{poly-2}) at $z=e_j$,
yields the linear constraints 
$$b_j\,=\,\sum_{k=1}^nA_{j;k}
\sum_{\genfrac..{0pt}1{u\in\nn^m,}{u\leq\beta-A_k}}y[k,u]
\quad\forall\;1\leq{j}\leq{m},$$
because $w^{e_j}=w_j$ for all $w\in\zz^m$. Constraints (\ref{m-a1}) 
are easily deduced after defining the vectors $A^{(e_j)}\in\nn^p$ by 
$A^{(e_j)}[k,u]:=A_{j;k}$ for all $u$ and $1\leq{k}\leq{n}$. We also have 
that $b=A\widehat{x}$ with $\widehat{x}=E\y$ and $E\in\nn^{n\times{p}}$ 
the matrix given in (\ref{matE}), e.g. $\widehat{x}_k=\sum_uy[k,u]$ for 
$1\leq{k}\leq{n}$. 

\bigskip
{\bf (b)} Let $\mathcal{P}\in\rr^p$ be the polytope given in 
(\ref{main-polytope}), and $H\subset\rr^n$ be the integer hull of 
$\{x\in\rr^n:Ax=b,x\geq0\}$. From Theorem~\ref{Lasserre} 
and the proof of Theorem~\ref{polynomial}, any nonnegative vector 
$\y\in\rr^p$ is a solution to (\ref{poly-2}) if and only if $\y\in\mathcal{P}$, because $\Theta\y=\bb$; and we have already stated in the 
proof of \textbf{(a)} that (\ref{m-a1})-(\ref{m-ab}) is completely 
equivalent to (\ref{poly-2}). We also deduced 
that $\widehat{x}=E\y$ is a solution to $A\widehat{x}=b$ for any feasible 
solution $\y\in\mathcal{P}$. Conversely, given any nonnegative integer 
solution $x\in\nn^n$ to $Ax=b$ and working as in the proof of 
Theorem~\ref{Lasserre}, we can associate with $x$ a nonnegative integer 
vector $\y(x)\in\nn^p$ such that $\Theta\y(x)=\bb$ and 
$x=E\y(x)$; hence $\y(x)\in\mathcal{P}$.

Let $\{\y^{[\ell]}\in\nn^p\}$ be the vertices of $\mathcal{P}$. Every 
$\y^{[\ell]}$ has nonnegative integer entries, because $\Theta$ is 
totally unimodular and the vector $\bb$ in (\ref{main-vector}) has 
integer entries. Each vector
$E\y^{[\ell]}\in\nn^n$ has nonnegative integer entries and 
satisfies $AE\y^{[\ell]}=b$, so that 
$E\y^{[\ell]}\in{H}$. Any feasible solution $\y\in\mathcal{P}$ can 
be expressed as the sum $\y=\sum_\ell\xi_\ell\y^{[\ell]}$ with 
$\sum_\ell\xi_\ell=1$ and $\xi_\ell\geq0$ for each $l$. Hence,
$$E\,\y\,=\sum_{\ell}\xi_{\ell}\,E\,\y^{[\ell]}\quad
\hbox{with}\quad{E}\y^{[\ell]}\in{H}\quad\forall\;\ell,$$
and so $\widehat{x}:=E\y\in H$. Conversely, let $\{x^{[c]}\in\nn^n\}$ be the vertices
of $H$. We claim that every $\y(x^{[c]})$ is a vertex of $\mathcal{P}$; 
otherwise if $\y(x^{[c]})=\sum_{\ell}\xi_{\ell}\y^{[\ell]}$ for 
$0\leq\xi_\ell<1$, then 
$$x^{[c]}=E\y(x^{[c]})=\sum_{\ell}\xi_{\ell}E\y^{[\ell]}
\quad\hbox{with}\quad{E}\y^{[\ell]}\in{H},$$
is a  combination of vertices $(E\y^{[\ell]})$ of $H$,
in contradiction with the fact that $x^{[c]}$ itself is a vertex of $H$.
\end{proof}

\subsection{A hierarchy of LP-relaxations}

Let $|z|:=z_1{+}\cdots{+}z_m$ for every $z\in\nn^m$. 
Consider the following integer and linear programming 
problems, for $1\leq\ell\leq|\beta|$,
\begin{eqnarray}
\label{ip}J_\bullet&:=&\min_x\,\{\,c'x\,:\,Ax=b\,,\,x\in\nn^n\,\};\\
\label{lpl}J_\ell&:=&\min_\y\bigg\{c'E\,\y:\,
\begin{matrix}b^z=\y'A^{(z)},\,\y\in\rr^p,\,\y\geq0,\\
z\in\nn^m,\,1\leq|z|\leq\ell\end{matrix}\bigg\}.
\end{eqnarray}

We easily have the following result.

\begin{corollary}\label{hierarchy}
Let $\beta,b\in\nn^m$ and $A\in\nn^{m\times{n}}$ be such that 
$\beta\geq{b}$ and $\beta\geq{A_k}$ for every index $1\leq{k}\leq{n}$. 
Let $J_\bullet$ and $J_\ell$ be as in (\ref{ip}) and (\ref{lpl}) 
respectively. 

Then~: $J_1\,\leq\,J_2\leq\ldots\leq\,J_{|\beta|}$ 
and $J_{|\beta|}\,=\,J_\bullet$.
\end{corollary}

\begin{proof}
It is obvious that $J_\ell\leq{J}_{\ell^*}$ 
whenever $\ell\leq\ell^*$. Moreover~:
$$J_\bullet=\min_x\{c'x:x\in{H}\}\quad\hbox{and}
\quad{}J_{|\beta|}=\min_\y\{c'E\y:\y\in\mathcal{P}\},$$
where $H$ is the integer convex hull of $\{x\in\rr^n:Ax=b,x\geq0\}$ 
and $\mathcal{P}$ is given in (\ref{main-polytope}). Thus 
$J_\bullet=c'x^{[d]}$ for some vertex $x^{[d]}$ of $H$. Working as 
in the proof of Corollary~\ref{main}, $\y(x^{[d]})$ is 
a vertex of $\mathcal{P}$ and $x^{[d]}=E\y(x^{[d]})$, so that 
$$J_{|\beta|}\,\leq\,c'E\,\y(x^{[d]})\,=\,J_\bullet.$$

On the other hand, $J_{|\beta|}=c'E\y^{[\kappa]}$ for 
some vertex $\y^{[\kappa]}$ of $\mathcal{P}$. Again, as in 
the proof of Corollary~\ref{main}, $E\y^{[\kappa]}\in{H}$,
so that
$$J_\bullet\,\leq\,c'E\,\y^{[\kappa]}\,=\,J_{|\beta|}.$$
\end{proof}

Hence the sequence of LP problems $\{J_\ell\}_{1\leq\ell\leq|\beta|}$ 
in (\ref{lpl}) provides a monotone sequence of lower bounds for 
$J_\bullet$ with finite convergence limit $J_{|\beta|}=J_\bullet$. 
Observe that all linear programs (\ref{lpl}) have the same number 
$p$ of variables (since $\y\in\rr^p$) and an increasing number of 
constraints starting from $m$ when $\ell=1$ to $s=\prod_j(\beta_j{+}1)$ 
when $\ell=|\beta|$. The most interesting case happens when we 
take $b=\beta\geq{A_k}$ (for every index $1\leq{k}\leq{n}$).

Of course $p\,(<ns)$ is large; but in principle the LP problems (\ref{lpl}) 
are amenable to computation for reasonable values of $\ell$ by using 
column generation techniques (in order to avoid handling the full 
vector $\y\in\rr^p$). However the vectors $A^{(z)}\in\nn^p$ may contain 
very large values when $|z|=\ell$ is large, and so some numerical 
ill-conditioned cases are likely to appear for large values of $\ell$.

Therefore it is not completely clear whether the hierarchy of 
linear relaxations presented in Corollary~\ref{main} is useful from a 
computational point of view, but we think this hierarchy provides some 
new insights into integer programming problems from a theoretical point 
of view.

\subsection{\bf A polynomial certificate.}

Let $\beta,b\in\nn^m$ and $A\in\nn^{m\times{n}}$ be such 
that $\beta\geq{b}$ and $\beta\geq{A_k}$ for every index 
$1\leq{k}\leq{n}$. Consider the matrices $\Delta\in\nn^{s\times{s}}$ and 
$\Theta\in\zz^{s\times{p}}$ given in (\ref{matrix-2})-(\ref{matrix-3}) 
of Definition~\ref{main-def}. Lemma~\ref{Kroneker} states that $\Delta$ 
is invertible, and so Theorem~\ref{Lasserre} implies that existence of 
a nonnegative integer solution $x\in\nn^n$ to $Ax=b$ is equivalent to 
the existence of a nonnegative real solution $\y\in\rr^p$ to 
\begin{equation}\label{eqn-21}
\Delta\Theta\,\y\,=\,\Delta\bb\quad\hbox{and}\quad\y\,\geq\,0,
\end{equation}
where $\bb\in\zz^s$ is given in (\ref{eqn+p3}) or (\ref{main-vector}). 
Moreover it was already stated in (\ref{poly-5}) that the entries of 
the product $\Delta\bb$ can be indexed with the nonnegative vector 
$z\in\nn^m$ with $z\leq\beta$, and that they have the simple 
representation ~:
\begin{equation}\label{eqn-22}
(\Delta\bb)[z]\,=\,(b^z-0^z)\,=\,\left\{\begin{array}{cl}
0&\hbox{if}\;z=0,\\ b^z&\hbox{if}\;z\neq0,\end{array}\right.
\end{equation}
using the conventions $0^0=1$ and $0^w=0$ for every $w\neq0$. 

To obtain a polynomial certificate we apply Farkas lemma to the linear programming problem 
(\ref{eqn-21}). 
Consider the polyhedral cone $C_\beta\subset\rr^s$ defined by 
\begin{equation}\label{coneC}
C_\beta\,:=\,\{\,\xi\in\rr^s\,:\,(\Delta\Theta)'\xi\geq0\,\}.
\end{equation}
With every element $\xi=(\xi_z)\in C_b$, we can associate a 
specific polynomial $p_\xi\in\rr[u_1,\ldots,u_m]$ defined by:
\begin{equation}
\label{pollike}
u\mapsto{p_\xi}(u):=\sum_{\genfrac..{0pt}1
{z\in\nn^m,}{z\neq0,\,z\leq\beta}}\xi_z\,u^z,\end{equation}
and notice that in view of (\ref{eqn-22}), $p_\xi(b)=\xi'\Delta\bb$.

\begin{theorem}\label{certificate}
Let $\beta,b\in\nn^m$ and $A=[A_1|\cdots|A_n]\in\nn^{m\times{n}}$ 
be such that $\beta\geq{b}$ and $\beta\geq{A_k}$ for every index 
$1\leq{k}\leq{n}$. Consider the polyhedral cone $C_\beta\subset\rr^s$ 
given in (\ref{coneC}) and the following two statements~:
\begin{description}
\item[(a)] The system $Ax=b$ has an integer solution $x\in\nn^n$.

\item[(b)] $p_\xi(b)<0$ for some $\xi\in{C_\beta}$ and $p_\xi\in\rr[u]$
as in (\ref{pollike}).
.
\end{description}

\noindent
Then one and only one of statements {\bf (a)} or {\bf (b)} is true.
\end{theorem}

\begin{proof}
The fact that $\Delta$ is invertible (because of Lemma~\ref{Kroneker})
and Theorem~\ref{Lasserre} imply that the system $Ax=b$ has a 
nonnegative integer solution $x\in\nn^n$ if and only if (\ref{eqn-21}) 
holds for some real nonnegative vector $\y\in\rr^p$. Farkas lemma 
implies that (\ref{eqn-21}) holds if and only if $\xi'\Delta\bb\geq0$ 
for every $\xi\in{C_\beta}$; however from (\ref{eqn-22}),
$$0\leq \xi'\Delta\bb\,=\!\sum_{\genfrac..{0pt}1{z\in\nn^m,}{z\neq0,\,z\leq\beta}}\xi_z\,b^z\,=\,
p_\xi(b),$$
the desired result.
\end{proof}

It is interesting to compare Theorem \ref{certificate} with
the standard Farkas lemma for (continuous) linear systems. The \textit{continuous} 
Farkas lemma provides a \textit{linear} certificate of the form 
$\omega'b<0$ for some $\omega\in\rr^m$ that satisfies $A'\omega\geq0$;
in contrast (the \textit{discrete}) Theorem~\ref{certificate} provides 
a \textit{non-linear} polynomial certificate $p_\xi(b)<0$.

\begin{example}
Let $A:=[3,4]\in\nn^2$, $\beta=5$, and $0\leq{b}\leq5$; i.e. the 
Frobenius equation $3x_1+4x_2=b$ with $x_1,x_2\in\nn$. Then $s=6$,
$$\Delta=\left[\begin{smallmatrix}
1&1&1&1&1&1\\ 
0&1&2&3&4&5\\ 
0&1&4&9&16&25\\ 
0&1&8&27&64&125\\ 
0&1&16&81&256&625\\ 
0&1&32&243&1024&3125
\end{smallmatrix}\right],
\quad\hbox{and}\quad
\Theta=\left[\begin{smallmatrix}
-1&0&0&-1&0\\ 
0&-1&0&0&-1\\ 
0&0&-1&0&0\\ 
1&0&0&0&0\\ 
0&1&0&1&0\\ 
0&0&1&0&1
\end{smallmatrix}\right].$$
The pair (\ref{eqn-21})-(\ref{eqn-22}) reads as follows~:
$$\begin{bmatrix}0&0&0&0&0\\ 
3&3&3&4&4\\ 
9&15&21&16&24\\
27&63&117&64&124\\ 
81&255&609&256&624\\ 
243&1023&3093&1024&3124
\end{bmatrix}\begin{bmatrix}
y[1,0]\\ y[1,1]\\ y[1,2]\\ y[2,0]\\ y[2,1]
\end{bmatrix}=\begin{bmatrix}
0\\ b\\ b^2\\ b^3\\ b^4\\ b^5
\end{bmatrix}.$$

Notice that $\Delta\Theta=\big[\genfrac..{0pt}10S\big]$ with 
$S\in\nn^{5\times{5}}$ an invertible matrix. Whence there is 
a nonnegative solution $\y\geq0$ if and only if
$$\y(b)\,=\,S^{-1}[b,b^2,b^3,b^4,b^5]'\,\geq\,0.$$
We can easily verify that $\y(5)=[0,-1,0,1,1]\not\geq0$. We also have that
$\y(b)\geq0$ for  $b=0,3,4$; and that $\y(b)\not\geq0$ for $b=1,2,5$. 
\end{example}

On the other hand, we can use the transpose $\Delta'$ instead of 
$\Delta$ in equations (\ref{eqn-21}) to (\ref{coneC}). Thus 
existence of a nonnegative integer solution $x\in\nn^n$ to 
$Ax=b$ is equivalent to existence of a nonnegative real 
solution $\y\in\rr^p$ to 
$$\Delta'\Theta\,\y\,=\,\Delta'\bb,\quad
\hbox{where}\quad(\Delta'\bb)[z]\,=\,(z^b-1).$$ 
Consider the polyhedral cone $C^*_\beta\subset\rr^s$ defined by 
$$C^*_\beta\,:=\,\{\,\xi\in\rr^s\,:\,\Theta'\Delta\,\xi\geq0\,\},$$
and with every $\xi=(\xi_z)\in{C}^*_\beta$,
associate the exponential-like function
\begin{equation}
\label{expolike}
\sum_{z\in\nn^m,\,z\leq\beta}\xi_z\cdot(z^u-1).\end{equation}
Observe that $f_\xi(b)=\xi'\Delta'\bb$. By Farkas lemma, one 
and only one of the following statements holds~:
\begin{description}
\item[(a)] The system $Ax=z$ has an integer solution $x\in\nn^n$.

\item[(b*)] $f_\xi(b)<0$ for some $\xi\in{C^*_\beta}$ and $f_\xi$ as in (\ref{expolike}).
\end{description}


{\small
\begin{thebibliography}{las}
\bibitem{DuFo}  
D. S. Dummit, R. M. Foote. {\em Abstract Algebra. 2nd edition}. 
John Wiley and Sons, New York, 1999.
\bibitem{gomoryjohnson}
R.E. Gomory, E.L. Johnson. The group problem and subadditive functions, 
in: {\em Mathematical Programming}, T.C. Hu and S.M. Robinson, editors.
Academic Press, New York, 1973.
\bibitem{hornjohnson}
R.A. Horn, C.R. Johnson. {\it Topics in Matrix Analysis}. 
Cambridge University Press, Cambridge, 1991.
\bibitem{johnson}
E.L. Johnson, {\it Integer Programming: Facets, 
Subadditivity, and Duality for Group and Semi-group Problems}, 
Society for Industrial and Applied Mathematics, Philadelphia, 1980.
\bibitem{lasserrefarkas}
J.B. Lasserre. A discrete Farkas lemma. 
{\em Discr. Optim.} {\bf 1} (2004),  67--75.  
\bibitem{lasserreduality}
J.B. Lasserre. Integer programming duality and superadditive functions. 
{\em Contemp. Math.} {\bf 374} (2005),  139--150.  
\bibitem{lasserresiopt1}
J. B. Lasserre. Global optimization with polynomials and the problem 
of moments. {\em SIAM J. Optim.} {\bf 11} (2001), 796--817. 
\bibitem{lasserresiopt2}
J. B. Lasserre. An explicit equivalent positive semidefinite program for 
nonlinear 0-1 programs. {\em  SIAM J. Optim.} {\bf 12} (2002), 756--769.
\bibitem{book}
J.B. Lasserre. {\it Linear and Integer Programming versus Linear Integration and Counting},
Springer, New York, 2009.
\bibitem{laurent}
M. Laurent. A comparison of the Sherali-Adams, Lov\'asz-Schrijver, 
and Lasserre relaxations for 0-1 programming.  
{\em Math. Oper. Res.}  {\bf 28}  (2003),  470--496. 
\bibitem{lovasz}
L. Lov\'asz and A. Schrijver. Cones of matrices and set-functions 
and 0-1 optimization. {\em SIAM J. Optim.} {\bf 1} (1991), 166--190.
\bibitem{sherali1}
H.D. Sherali and W.P. Adams. A hierarchy of relaxations between the 
continuous and convex hull representations for zero-one programming 
problems. {\em SIAM J. Discr. Math.} {\bf 3} (1990), 411--430. 
\bibitem{sherali2}
H.D. Sherali and W.P. Adams. {\it A reformulation-linearization 
technique for solving discrete and continuous nonconvex problems}. 
Kluwer Academic Publishers, Dordrecht, 1999.
\bibitem{schrijver}
A. Schrijver. {\it Theory of Linear and Integer Programming}.
John Wiley \& Sons, Chichester, 1986.
\bibitem{Shevchenko}
V.N. Shevchenko. {\it Qualitative Topics in Integer Linear Programming},
Translations of Mathematical Monographs, volume 156. 
American Mathematical Society, Providence, 1997.
\bibitem{wolsey1}
L.A. Wolsey. Integer programming duality:  Price functions and
sensitivity analysis. {\em Math. Program.} {\bf 20} (1981), 173--195.
\bibitem{zhang}
F. Zhang. {\it Matrix Theory, Basic Results and Techniques}.
Springer-Verlag, New York, 1999.
\end{thebibliography}}
\end{document}